\documentclass{amsart}
\usepackage[utf8]{inputenc}
\usepackage[lite]{amsrefs}
\usepackage{amssymb}
\usepackage{stmaryrd}
\usepackage{enumitem}
\usepackage[english]{babel}
\usepackage{amstext} 
\usepackage{array}
\usepackage{longtable}

\newcolumntype{L}{>{$}c<{$}}

\usepackage{graphicx}

\DeclareMathOperator{\codim}{codim}
\DeclareMathOperator{\supp}{supp}
\DeclareMathOperator{\defect}{def}
\DeclareMathOperator{\proj}{pr}
\DeclareMathOperator{\GL}{GL}
\DeclareMathOperator{\rank}{rk}
\DeclareMathOperator{\noneq}{np}
\DeclareMathOperator{\dom}{dom}

\newcommand{\N}{\mathcal{N}}
\newcommand{\Nn}{\mathcal{N}} 

\numberwithin{equation}{section}
\theoremstyle{plain} 
\newtheorem{thm}[equation]{Theorem}
\newtheorem{cor}[equation]{Corollary}
\newtheorem{lem}[equation]{Lemma}
\newtheorem{prop}[equation]{Proposition}

\theoremstyle{definition}

\theoremstyle{remark}
\newtheorem{rem}[equation]{Remark}

\newtheorem{ex}[equation]{Example}

\title{Closure relations of Newton strata in Iwahori double cosets}
\author{Stefania Trentin}
\author{Eva Viehmann}
\address{Technische Universitat M\"unchen, Fakult\"at f\"ur Mathematik - M11, Boltzmannstr. 3,
85748 Garching bei M\"unchen, Germany}
\thanks{This work was partially supported by ERC Consolidator Grant 770936:\ NewtonStrat and by a PreDoc scholarship of the Faculty of Mathematics at the Technical University of Munich.}

\begin{document}

\begin{abstract}
We consider the Newton stratification on Iwahori double cosets for a connected reductive group. We prove the existence of Newton strata whose closures cannot be expressed as a union of strata, and show how this is implied by the existence of non-equidimensional affine Deligne-Lusztig varieties. We also give an explicit example for a group of type $A_4$.
\end{abstract}

\maketitle


\section{Introduction}
Let $G$ be a connected reductive group defined over $ F = \mathbb{F}_q(\!(t)\!)$ which for the purpose of this introduction is assumed to be split.

Let $\breve F = \overline{\mathbb{F}}_q(\!(t)\!)$ and  $\mathcal{O}_{\breve F}$ its ring of integers. Let $\sigma$ be the morphism on $G({\breve F})$ induced by the Frobenius automorphism of ${\breve F}$ over $F$.

For any element $b$ of $G({\breve F})$ we consider its $\sigma$-conjugacy class  \begin{equation*}
    [b] = \{g^{-1}b\sigma(g) \mid g \in G({\breve F})\},
\end{equation*}
and the Kottwitz set of $\sigma$-conjugacy classes $B(G)=\{[b]\mid b\in G(\breve F)\}$. For more details and notation we refer to  Section \ref{sec:2}.  The subdivision of $G(\breve F)$ into $\sigma$-conjugacy classes corresponds to a subdivision of the loop group $LG$ into Newton strata $\N_{[b]}$, locally closed subsets whose $\bar{\mathbb{ F}}_q$-valued points agree with the respective class $[b]$. By results of Rapoport-Richartz \cite{RR}*{Thm.~3.6} and Viehmann \cite{V13}, the closure of a Newton stratum is given by
\begin{equation}\label{eqcl1}
\overline{\mathcal N_{[b]}}=\bigcup_{[b']\leq [b]}\mathcal N_{[b']},
\end{equation}
 where $\leq$ is a partial order on $B(G)$. The index set on the right hand side is finite. 

For applications to the geometry of the special fiber of moduli spaces of shtukas, or of Shimura varieties, one wants to understand the relation between this stratification and double cosets under parahoric subgroups of $G$.

We begin by recalling the case of a hyperspecial maximal parahoric subgroup $K$ of $G$. We fix a maximal torus $T$ of $G$ and a Borel subgroup $B$ containing it. By the Cartan decomposition we have 
\begin{equation*}
     G({\breve F}) = \coprod_{\mu\in X_*(T)_{\dom}} K(\mathcal{O}_{\breve F})\mu(t)K(\mathcal{O}_{\breve F}).
\end{equation*}
Then $\N_{[b],\mu}:=\mathcal N_{[b]}\cap K(\mathcal{O}_{\breve F})\mu(t)K(\mathcal{O}_{\breve F})\neq\emptyset$ if and only if $[b]\leq [\mu(t)]$. The closure relations within $K(\mathcal{O}_{\breve F})\mu(t)K(\mathcal{O}_{\breve F})$ are as in \eqref{eqcl1}, i.e. \begin{equation}\label{eqcl2}
\overline{\mathcal N_{[b],\mu}}=\bigcup_{[b']\leq [b]}\mathcal N_{[b'],\mu}
\end{equation} whenever $\overline{\mathcal N_{[b],\mu}}$ is non-empty, compare \cite{V13}. Furthermore, $\N_{[b],\mu}$ is pure of codimension equal to the length of any maximal chain between $[b]$ and $[\mu]$ in $B(G)$.

We now describe the situation in the case we are interested in for this work. Let $I$ be an Iwahori subgroup of $G$. Then the affine Bruhat decomposition gives \begin{equation*}
     G({\breve F}) = \coprod_{x \in \widetilde{W}} I(\mathcal{O}_{\breve F})xI(\mathcal{O}_{\breve F}),
\end{equation*} where $\widetilde{W}$ denotes the extended affine Weyl group.

Again, we define Newton strata in a given double coset
by  $$\mathcal{N}_{[b], x} := [b] \cap IxI.$$ They are equipped with the structure of locally closed reduced subschemes of $IxI$. However, it turns out to be a difficult task to describe which $\N_{[b],x}$ are non-empty. For some results in this direction compare \cite{GHN} and \cite{He20}. Let $$B(G)_x:=\{[b]\in B(G)\mid \N_{[b],x}\neq\emptyset\}.$$ Contrary to the behaviour in the hyperspecial case, there are $x$ and $[b_1]\leq[b_2]\leq [b_3]\in B(G)$ with $[b_1],[b_3]\in B(G)_x$, but $[b_2]\notin B(G)_x$. In other words the subset $B(G)_x\subset B(G)$ can be non-saturated. Also, the Newton strata are in general no longer of codimension equal to the length of any maximal chain between $[b]$ and the $\sigma$-conjugacy class $[b_x]$ corresponding to the generic point of $IxI$.

Using an example of a non-equidimensional affine Deligne-Lusztig variety given in \cite{gor}*{5} together with \cite{Vimi}*{Cor.~3.11}, one sees that in general, the $\N_{[b],x}\subseteq IxI$ are no longer pure of any fixed codimension.

In this paper we study the closures of Newton strata in a given Iwahori double coset. By \eqref{eqcl1} we obtain 
\begin{equation}\label{eqcl3}
\overline{\N_{[b],x}}\subseteq \N_{\le [b], x} = \bigcup_{[b'] \le [b]} \N_{[b'], x}.
\end{equation}
 Besides this, very little was known previously. On the one hand, there are examples of double cosets where the whole pattern of closures is known (in particular for the group ${\rm SL}_3$ by \cite{B} and for so-called cordial elements by \cite{Vimi}). In all of these examples, we still have equality in \eqref{eqcl3} for all non-empty $\N_{[b],x}$. On the other hand, the difficult behaviour concerning non-emptiness and dimensions lets one suspect that this does not hold in general.
 
In Theorem \ref{lem:1}, we confirm this suspicion. We show that if there is a Newton stratum $\N_{[b],x}$ which is not pure of some codimension in $IxI$ (or equivalently an affine Deligne-Lusztig variety $X_x(b)$ which is not equidimensional), then for every maximal chain $[b]<[b_1]<\dotsm <[b_m]=[b_x]$ in $B(G)_x$ there is an $1\leq i<m$ such that the containment in \eqref{eqcl3} is strict for $[b_i]$. Also, we prove that under the same assumption there is a $[b']> [b]$ such that $\overline{\N_{[b'],x}}$ is not a union of strata. Thus in such cases, the decomposition of $IxI$ into its intersections with the various $\sigma$-conjugacy classes  is in general not a stratification (although we continue to call the subschemes $\N_{[b],x}$ Newton strata). The proof of Theorem \ref{lem:1} relies on topological strong purity, which we review in Section \ref{sec:2}.

In Section \ref{sec:4} we construct an explicit pair $([b],x)$ for $G$ of type $A_4$, such that the closure in $IxI$ of $\mathcal{N}_{[b], x}$ is not a union of strata. For this, we use the reduction method \`a la Deligne and Lusztig as proved in \cite{he} to reduce the computation of dimensions of Newton strata in $IxI$ to similar computations for cordial elements $x'$ which then follow from the general theory of cordial elements as in \cite{Vimi}, compare also Section \ref{sub:3}. Our approach is inspired  by the construction of a non-equidimensional affine Deligne-Lusztig variety for $[b]=[1]$ given by G\"ortz and He in \cite{gor}*{Sec.~5}, but adapted to find a maximal class for which the corresponding Newton stratum is not equi-dimensional. For the family of elements $x$ in the affine Weyl group that we obtain in this way, we can still give a complete description of the subset $B(G)_x$, as well as of all codimensions of irreducible components of Newton strata. 

\subsection*{}\emph{Acknowledgments.}  We thank Felix Schremmer for his help on using \verb+SageMath+. 

\section{Notation and review of previous results}\label{sec:2}
In this section we fix the notation and recall the necessary theory of Newton strata and affine Deligne-Lusztig varieties.

\subsection{Notation}\label{sec21}
Let $G$ be a connected reductive group over the local field $F \cong \mathbb{F}_q(\!(t)\!)$, where $\mathbb{F}_q$ is the finite field with $q$ elements. Let ${\breve F} \cong \overline{\mathbb{F}}_q(\!(t)\!)$ be the completion of the maximal unramified extension of $F$, and let $\sigma$ be the Frobenius automorphism of ${\breve F}$ over $F$, mapping all coefficients to their $q^{\text{th}}$ powers. We also denote the induced automorphism of $G({\breve F})$ by $\sigma$. Let $\Gamma$ be the absolute Galois group of $F$.

Fix $S$, a maximal ${\breve F}$-split torus in $G$ defined over $F$ and containing a maximal $F$-split torus. Let $T$ be the centralizer of $S$ in $G$, a maximal torus. Consider the apartment $\mathcal{A}$ of $G_{\breve F}$ associated to $S_{\breve F}$. The Frobenius $\sigma$ acts on $\mathcal{A}$ and we fix a $\sigma$-stable alcove $\mathfrak{a}$. Let $I \subset G({\breve F})$ be the Iwahori subgroup corresponding to $\mathfrak{a}$.

Let $N_T$ be the normalizer of $T$ in $G$. The (relative) Weyl group $W_0$ is defined to be $W_0 = N_T({\breve F})/T({\breve F})$, and the extended affine Weyl group is $\widetilde{W} = N_T({\breve F})/(T({\breve F})\cap I)$. Choosing a vertex of $\mathfrak{a}$ determines a section $W_0 \hookrightarrow \widetilde{W}$ of the natural projection map $\widetilde W\rightarrow W_0$. Since $T$ and $I$ are $\sigma$-stable, the Frobenius induces automorphisms of $W_0$ and $\widetilde W$, which we also denote by $\sigma$.

For every $x \in \widetilde{W}$ we choose a representative in $G(\breve F)$, which we denote again by $x$. We consider the affine Bruhat decomposition \begin{equation*}
    G({\breve F}) = \coprod_{x \in \widetilde{W}} I(\mathcal{O}_{\breve F})xI(\mathcal{O}_{\breve F}).
\end{equation*}

To define a length function on the extended affine Weyl group, we consider the decomposition $\widetilde{W} \cong \Omega \ltimes W_a$, where $W_a$ is the affine Weyl group of $G$ and $\Omega$ is the subset of elements of $\widetilde{W}$ which fix the chosen Iwahori subgroup. We extend the length function $\ell$ from $W_a$ to $\widetilde{W}$ by setting $\ell(\omega) = 0$ for $\omega \in \Omega$. 

By convention, the dominant Weyl chamber is opposite to the unique chamber containing $\mathfrak{a}$. This choice determines a set of simple roots $\Delta$. Consider $\mathbb{S}$,  the set of simple reflections of $W_0$. For any $i \in \mathbb{S}$, we denote by $\alpha_i \in \Delta$ the corresponding simple root.

For any subset $J \subset \mathbb{S}$, let $W_{J}$ be the subgroup of $W_a$ generated by $J$. Denote by $^J\widetilde{W}$ the set of minimal length representatives of the cosets $ W_{J} \backslash \widetilde{W}$. For a coweight $\mu \in X_{*}(T)$ we denote by $t^{\mu} \in T({\breve F})$ the image of $t$ under $\mu$.

Any element of the extended affine Weyl group can be written in a unique way as $vt^{\mu}w$ with $\mu \in X_{*}(T)$, $v$ and $w$ in $W_0$ and and $t^{\mu}w $ in $^{\mathbb{S}}\widetilde{W}$.

\subsection{$\sigma$-conjugacy classes and Newton strata}

For $b \in G({\breve F})$ let $[b]=\{g^{-1}b\sigma(g) \mid g \in G({\breve F})\}$ denote its $\sigma$-conjugacy class, and let $B(G)=\{[b]\mid b\in G(\breve F)\}$. By \cite{Ko}, \cite{Ko''} and \cite{RR}, the elements of $B(G)$ are determined by two classifying invariants, the Newton point $\nu([b])$ and the Kottwitz point $\kappa_G([b])$. 

For any element $b \in G({\breve F})$, we have a (Newton) homomorphism $\nu(b): \mathbb{D} \rightarrow G$, where $\mathbb{D}$ is the pro-algebraic torus with character group $\mathbb{Q}$. Its $G$-conjugacy class is the first invariant of $[b]$, called the Newton point. We often identify it with its representative in $(X_*(T)_{\mathbb{Q},\dom})^{\Gamma}$.

The second invariant is given by the image under the Kottwitz map $\kappa_G: B(G) \rightarrow \pi_1(G)_{\Gamma}$, where the fundamental group $\pi_1(G)$ is defined as the quotient of $X_{*}(T)$ by the coroot lattice. The Kottwitz map is induced by a homomorphism $G({\breve F})\rightarrow \pi_1(G)_{\Gamma}$. Furthermore, it is trivial on $I$. In particular, $\kappa_G(b)=\kappa_G(x)$ for every $b\in IxI$. Thus when restricting to $IxI$, the Newton point alone is sufficient to uniquely determine a class. 

Two elements in $X_{*}(T)_{\mathbb{Q}, \dom}$ satisfy $\nu_1 \le \nu_2$ in the dominance (partial) order, if and only if their difference $\nu_2 - \nu_1$ is a linear combination of positive coroots with non-negative, rational coefficients. This induces a partial order on the set $B(G)$: we have $[b] \le [b']$ if $\kappa_G(b)=\kappa_G(b')$ and $\nu(b) \le \nu(b')$. This equips $B(G)$ with the structure of a ranked lattice. For any two classes $[b] \le [b']$, every maximal chain in $B(G)$ from $[b]$ to $[b']$ has the same length, compare \cite{chai}*{Thm.~7.4}.

For $G = \GL_n$ we have an interpretation of these notions in terms of the Newton polygon. Let $T \subset B$ be the subgroups of diagonal, respectively of upper triangular matrices. Then $I$ is the subgroup of $\GL_n(\mathcal{O}_{\breve F})$, whose image modulo $t$ is the Borel subgroup opposite to $B$. In this case $X_{*}(T)_{\mathbb{Q}} \cong \mathbb{Q}^n$  and $\nu = (\nu_i)$ is dominant, if and only if $\nu_i \ge \nu_{i+1}$ for all $i$. The Newton point $\nu$ of a class $[b]$ coincides with the classical Newton point of the isocrystal $({\breve F}^n, b \sigma)$. Let $p_{\nu}$ be the polygon associated with $\nu$, \textit{i.e.}~the graph of the continuous piecewise linear function $[0, n] \rightarrow \mathbb{R}$ mapping $0$ to $0$ and of slope $\nu_i$ on the interval $[i-1, i]$. In this case, the subset of elements of $X_{*}(T)_{\mathbb{Q}}$ actually appearing as Newton points consists of the dominant $\nu \in \mathbb{Q}^n$, whose associated Newton polygon $p_{\nu}$ has break points and end point with integral coordinates.

\subsection{Newton strata in Iwahori double cosets}\label{sec_nsi}
Our main objects of study are Newton strata in an Iwahori double coset. These are defined for $x \in \widetilde{W}$ and $[b] \in B(G)$ as \begin{equation*}
    \mathcal{N}_{[b], x} = [b] \cap IxI.
\end{equation*} 
Further, let $$B(G)_x=\{[b]\in B(G)\mid \N_{[b],x}\neq\emptyset\},$$ a finite subset of $B(G)$.

Recall that the loop group associated with $G$ is the ind-group scheme representing the functor on $\mathbb{F}_q$-algebras $R \mapsto G(R(\!(t)\!) )$. 
When non-empty, by \cite{RR}*{Thm.~3.6} a Newton stratum $\N_{[b],x}$ is the set of $\overline{\mathbb{F}_q}$-valued points of a locally closed subset of $IxI$ (and thus of $LG$), which we equip with the structure of a reduced subscheme.

Recall that a subscheme $Z$ of $LG$ is bounded if it is contained in a finite union of double cosets $IxI$. Let $I_n$ be the kernel of the projection map $I \rightarrow I(\mathcal{O}_{\breve F}/(t^n))$. Then $Z$ is called admissible if there is an $n \in \mathbb{N}$ with $ZI_n = Z$. By \cite{Vimi}*{Prop.~3.5} every Newton stratum $\N_{[b],x}$ is admissible. Since $B(G)_x$ is finite, there also is an $n \in \mathbb{N}$ with $\N_{[b],x}I_n = \N_{[b],x}$ for every $[b]\in B(G)_x$.

We can then define the {codimension} of $\mathcal{N}_{[b], x}\subset IxI$ as the codimension of its image in $IxI/I_n$. This is independent of the choice of $n$, compare \cite{Vimi}*{Rem~3.4}. Similarly, the {closure} $\overline{{\mathcal{N}}_{[b],x}}$ of a Newton stratum in $IxI$ is the preimage of its closure in the quotient $IxI/I_n$. 

For a class $[b]$, we also consider $$ \mathcal{N}_{\le [b], x}=\bigcup_{[b']\leq [b]}\N_{[b'],x}.$$ The {specialization theorem} of \cite{RR}*{Thm.~3.6} implies that this set is closed in $IxI$ and contains $\mathcal{N}_{[b], x}$ as an open subset.

Since Iwahori double cosets are irreducible, for $x \in \widetilde{W}$ we denote by $[b_x]$ the $\sigma$-conjugacy class in the generic point of $IxI$, and its Newton point $\nu_x$. By \eqref{eqcl1}, the class $[b_x]$ is the unique maximal element of $B(G)_x$.

By \cite{ham}*{Sec.~2}, the Newton stratification on  $IxI/I_n$ satisfies \emph{topological strong purity}. Considering inverse images under the projection map $IxI\rightarrow IxI/I_n$ we obtain that for any $[b] \in B(G)$, any maximal element of $\{[b'] \in B(G)\mid [b'] < [b]\}$ and every irreducible component $Z$ of $\mathcal{N}_{\le [b], x}$, the subscheme $Z_{\le [b']} = Z \cap \mathcal{N}_{\le [b'], x} \subset Z$ is either empty or pure of codimension at most $1$ in $Z$.

\subsection{Dimensions and cordial elements}\label{sub:3}
Newton strata in Iwahori double cosets are closely related to the study of affine Deligne-Lusztig varieties. These were introduced by Rapoport in \cite{rap} and are defined for $x \in \widetilde{W}$ and $b \in G({\breve F})$ as the locally closed reduced subschemes of the affine flag variety for $G$ with
\begin{equation*}
    X_x(b)(\overline{\mathbb{F}}_q) = \{ g \in G({\breve F})/I(\mathcal{O}_{\breve F}) \mid g^{-1}b\sigma(g) \in IxI\}.
\end{equation*}
By \cite{Vimi}*{Cor.~3.12}, we have 
\begin{equation}\label{eq:adlvdim}
    \dim X_x(b) = \ell(x) - \langle 2\rho, \nu(b) \rangle - \codim (\mathcal{N}_{[b], x}),
\end{equation} 
where $\rho$ denotes the half-sum of the positive roots. Moreover, $X_x(b)$ is equidimensional if and only if $\mathcal{N}_{[b], x}$ is equicodimensional in $IxI$, see \cite{Vimi}*{Cor.~3.11}.

We recall the following invariants. Let $x\in\widetilde{W}$ and $[b] \in B(G)$. 
\setlist{nolistsep}\begin{itemize}[noitemsep]
    \item[(1)] Write $x = vt^{\mu}w$ with $v,w\in W_0$, $\mu\in X_*(T)$ and such that $t^{\mu}w\in {}^{\mathbb S}\tilde W$. Then $\eta: \widetilde{W} \rightarrow W_0$ maps $x$ to $\eta(x) = \sigma^{-1}(w)v$. 
    \item[(2)] The \emph{defect} of $b$ is $\defect(b) = \rank_FG - \rank_FJ_b $, where $J_b$ is the reductive group over $F$ with $J_b(F) = \{g \in G({\breve F}) \mid gb = b \sigma(g)\}$ \cite{kot6}*{Lem. 1.9.1}.
    \item[(3)] The \emph{virtual dimension} of the pair $(x, b)$ (as in \cite{he}*{Sec. 10.1}) is \begin{equation*}
        d_x(b) = \frac{1}{2} \Big ( \ell(x) + \ell(\eta(x))  - \defect(b) - \langle 2\rho, \nu(b) \rangle \Big).
    \end{equation*} 
\end{itemize}

By \cite{he2}*{Thm. 2.30} we have \begin{equation*}
    \dim X_x(b) \le d_x(b).
\end{equation*}

Let again $[b_x]$ denote the generic class in $IxI$. If the dimension of $X_x(b_x)$ agrees with its virtual dimension, the same holds by \cite{Vimi}*{Cor.~3.17} for all classes in $B(G)_x$. An equivalent condition is $\ell (x) - \ell (\eta(x)) = \langle 2\rho, \nu_x \rangle - \defect(b_x)$, and elements $x$ satisfying this are called cordial elements. Assume that $x$ is cordial. Then by \cite{Vimi}*{Thm.~1.1}
$B(G)_x$ is saturated, i.e.~for any $[b] \le [b'] \le [b'']$, if $[b]$ and $[b'']$ are in $B(G)_x$ then also $[b'] \in B(G)_x$. Further, for any $[b] \in B(G)_x$ we have \begin{itemize}
\item $\N_{[b], x}\subseteq IxI$ is pure of codimension equal to the maximal length of a chain from $[b]$ to $[b_x]$ in $B(G)$.
\item $\overline{\N_{[b], x}}=\bigcup_{[b']\le [b]}\N_{[b'],x}$.
\end{itemize}
In other words, neither of the two phenomena we are interested in arises for cordial $x$.

\section{Non-equidimensional strata and closure relations}\label{sec:3}
In this section we show that the existence of non-equicodimensional Newton strata implies that the Newton stratification fails to be a stratification.

\begin{thm}\label{lem:1}

Let $x\in \widetilde{W}$ and assume that there is a $[b]\in B(G)_x$ such that the corresponding Newton stratum $\mathcal{N}_{[b], x}$ has irreducible components of different codimensions in $IxI$.
\begin{enumerate}
\item For any maximal chain 
\begin{equation*}
   [b] = [b_0] < [b_1] < \dots <[b_m] = [b_x]
\end{equation*}
in $B(G)_x$, there is an $ 0 < i < m$ with $\overline{\N_{[b_i], x}}\neq \bigcup_{[b']\leq [b_i]}\N_{[b'],x}$.
\item There are $[b'']>[b']\geq[b]$ such that $\N_{[b'],x}\supsetneq\overline{\N_{[b''], x}}\cap \N_{[b'],x}\neq \emptyset$.
\end{enumerate}

\end{thm}
\begin{rem}
We do not know if every maximal chain in $B(G)_x$ has the same length.

In \cite{gor}*{5.2}, G\"ortz and He construct a non-equidimensional affine Deligne-Lusztig variety (for $G$ of type $A_3$ and $b=1$). By (\ref{eq:adlvdim}), the irreducible components of the corresponding Newton stratum  do not all have the same codimension. Then Theorem \ref{lem:1} implies that in this case, the Newton stratification is no longer a stratification. For another such example compare Section \ref{sec:4} below.
\end{rem}

\begin{proof}[Proof of Theorem \ref{lem:1}]
Let $\N_0$ be an irreducible component of $\mathcal{N}_{[b],x}$ and let $n_0$ be its codimension in $IxI$. 

Let $n$ be such that all Newton strata in $IxI$ are invariant under $I_n$. Recall that $\mathcal{N}_{[b],x}/I_n$ has finitely many irreducible components, as it is a locally closed subscheme of the Noetherian scheme $IxI/I_n$, and hence is itself Noetherian. In particular, the closure of $\N_0$ in $IxI$ is an irreducible component of the closure of $\N_{[b],x}$ and of the same codimension $n_0$ in $IxI$. 

To prove (1) we assume that there is a maximal chain of classes in $B(G)_x$ \begin{equation}\label{eq:1b}
    [b] = [b_0] < [b_1] < \dots <[b_m] = [b_x],
\end{equation}
such that $\overline{\N_{[b_i], x}}= \bigcup_{[b']\leq [b_i]}\N_{[b'],x}$ for every $i>0$. Inductively, one can then construct a chain of closed irreducible subschemes of $IxI$ \begin{equation}\label{eq:2}
    \overline{\N}_0 = Z_0 \subseteq Z_1 \dots \subseteq Z_m = IxI,
\end{equation} such that each $Z_i$ is an irreducible component of $\N_{\le [b_i],x}$ with generic $\sigma$-conjugacy class $[b_i]$. 

 From the topological strong purity of the Newton stratification (compare the end of Section \ref{sec_nsi}) and the maximality of (\ref{eq:1b}), it follows that $Z_{i-1}$ has codimension at most $1$ in $Z_{i}$. Since the generic $\sigma$-conjugacy classes of $Z_{i-1}$ and $Z_i$ are not equal, the two closed irreducible subschemes are not equal, and hence the codimension is equal to 1. Altogether, the codimension $n_0$ of $\Nn_0$ is equal to $m$, which is independent of the choice of the irreducible component $\Nn_0$. Thus each irreducible component of $\N_{[b],x}$ has the same codimension.
 
We now prove (2). Possibly replacing $[b]$ by a larger class, we may assume that all $\N_{[b'],x}$ with $[b]<[b']$ are equicodimensional. Let $\Nn_0$ be an irreducible component of $\N_{[b],x}$, and let $Z_0$ be its closure in $IxI$. By the same argument as above we see that $Z_0$ is an irreducible component of $\overline{\N_{[b],x}}$. Let $[\tilde b] >[b]$ be such that $ Z_0\subset \overline{\mathcal{N}_{[\tilde b], x}}$. We choose $[\tilde b]$ minimal in $B(G)_x$ with respect to this property, meaning that $Z_0$ is not contained in $\overline{\mathcal{N}_{[b''], x}}$ for any $[b] < [b''] < [\tilde b]$. By the Purity Theorem together with the minimality of $[\tilde b]$, we obtain that $Z_0$ has codimension $1$ in any irreducible component of $\overline{\mathcal{N}_{[\tilde b], x}}$ containing it. By the maximality of $[b]$ we obtain that $\N_{[\tilde b],x}$ is pure of some codimension $\tilde n$ in $IxI$, hence $\Nn_0$ has codimension $\tilde n+1$ in $IxI$.

Assume that the assertion of (2) does not hold for $[b']=[b]$, that is for every $[b'']>[b]$, the closure of $\N_{[b''],x}$ either contains $\N_{[b],x}$ or has empty intersection with it. Then $[\tilde b]$ can be chosen to be the same for every irreducible component $\Nn_0$ of $\N_{[b],x}$, which implies that $\N_{[b],x}$ is pure of codimension $\tilde n+1$, contradiction.
\end{proof}

\begin{cor}\label{cor:max}
Let $[b] \in B(G)_x$ such that $[b]$ is maximal in $B(G)_x \setminus [b_x]$. Then $X_x(b)$ is equidimensional of dimension $\ell(x) - \langle 2\rho, \nu(b) \rangle -1$.
\end{cor}
\begin{proof}
Since $\N_{[b_x],x}$ is dense in $IxI$, its closure contains $\N_{[b],x}$. By our assumption $[b] < [b_x]$ is a maximal chain in $B(G)_x$, and for $[b_x]$, equality holds in \eqref{eqcl3}. Thus by the proof of Theorem \ref{lem:1} (1), $\N_{[b], x}$ is pure of codimension 1 in $IxI$. By (\ref{eq:adlvdim}), this is equivalent to $X_x(b)$ being equidimensional of the claimed dimension.
\end{proof}

\section{An explicit example}\label{sec:4}

In the previous section we related the existence of Newton strata whose closures are not a union of strata to the existence of non-equidimensional strata, and used this to show that the Newton stratification is in general not a stratification in the proper sense. In this section we work out an explicit example of a Newton stratum in some $IxI$ whose closure is not a union of strata, generalizing the method of \cite{gor}*{5.2}. Despite the failure to be a stratification, we can in this example determine which Newton strata in $IxI$ are non-empty and compute all codimensions of irreducible components. We expect that this serves as a guiding example of typical behaviour of the Newton stratification also in other, more general cases.

\subsection{A family of pairs $(x,s)$ in $ W_a \times \mathbb{S}$}

Our goal is to find an $x$ and an $\N_{[b],x}$ having irreducible components of different codimensions. For this we generalize the method of \cite{gor}*{5}, where an $x\in W_a$ is constructed such that the variety $X_x(1)$ is not equidimensional. However, we consider larger $[b]$, and rather want to minimize the length of maximal chains between $[b]$ and $[b_x]$, to make the closure relations easier to handle. Also, we do not restrict ourselves to split groups as in loc.~cit.

\begin{lem}\label{lem:neq}
Let $x = vt^{\mu}w\in \tilde W$ with $v,w\in W_0$ and $t^{\mu}w\in {}^{\mathbb S}\widetilde W$, and let $s \in \mathbb{S}$ satisfying the following conditions:
\begin{enumerate}
	\item $sx\sigma(s)$ lies in the shrunken Weyl chamber
    \item $\ell(sv) < \ell(v)$ and $\ell(w\sigma(s)) > \ell(w)$
    \item $\ell(\sigma^{-1}(w)sv) < \ell(\sigma^{-1}(w)v)- 1$
    \item $\supp_{\sigma}(\sigma^{-1}(w)v) = \supp_{\sigma}(\sigma^{-1}(w)sv) = \mathbb{S}$ where $\supp_{\sigma}(u)=\bigcup_{i}\sigma^i(\supp(u))$ for any $u\in W_0$.
\end{enumerate}

Then $\ell(sx\sigma(s)) = \ell(x) -2$, and $sx$ and $x$ also lie in the shrunken Weyl chamber. Let $[b]\in B(G)_x$. Then
\begin{enumerate}
    \item[(i)]  $\dim X_x(b) = \max\{ \dim X_{sx}(b), ~ \dim X_{sx\sigma(s)}(b)\} + 1$. 
    \item[(ii)] If $\dim X_{sx\sigma(s)}(b)> \dim X_{sx}(b)$ and both are non-empty, then $X_x(b)$ has irreducible components of dimension $\dim X_{sx\sigma(s)}(b)+1$ and $\dim X_{sx}(b)+1$. In particular, it is not equidimensional.
    \item[(iii)] $ d_{sx\sigma(s)}(b) > d_{sx}(b).$
\end{enumerate}
\end{lem}
\begin{proof}
Condition (2) implies that $\ell(sx\sigma(s)) = \ell(x) -2$, and this together with (1) shows that $sx$ and $x$ are also in the shrunken Weyl chamber. 

By the Deligne-Lusztig reduction method (compare \cite{he}*{Prop.~4.2}), $X_x(b)$ can be written as a disjoint union $X_x(b) = X_1 \sqcup X_2$ of a closed subscheme $X_1$ and an open subscheme $X_2$, where $X_1$ is of relative dimension one over $X_{sx\sigma(s)}(b)$ and $X_2$ is of relative dimension one over $X_{sx}(b)$. In particular, this proves (i) and (ii).

Let $b_0$ be the unique basic element with $\kappa_G(b_0)=\kappa_G(x)$. Using that $x,sx,sx\sigma(s)$ are shrunken and Condition (4), \cite{he}*{Cor.~12.2} implies that $X_x(b_0)$, $X_{sx}(b_0)$ and $X_{sx\sigma(s)}(b_0)$ are non-empty and that their dimensions agree with the respective virtual dimensions $d_x(b_0)$, $d_{sx}(b_0)$ and $d_{sx\sigma(s)}(b_0)$. Note that loc.~cit.~only considers the case where $\sigma$ acts trivially on $W_0$. For the generalization to the present case compare the remarks following \cite{HeNote}*{Thm.~5.3}. Furthermore, (3) and \cite{gor}*{Lem.~3.2.2} imply that $d_x(b_0) > d_{sx}(b_0) +1 $. Again, loc.~cit.~only considers the split case. However, the proof carries over to our situation if one makes the obvious adaptations due to the action of $\sigma$ on $W_0$. From these observations and the reduction method we can deduce \begin{equation*}
   \dim X_2 = \dim X_{sx}(b_0) +1 = d_{sx}(b_0) + 1 < d_x(b_0) = \dim X_x(b_0).
\end{equation*}
By (i), this also shows that $d_{sx\sigma(s)}(b_0) > d_{sx}(b_0)$. For any $[b]\in B(G)$ we have \begin{equation*}
    d_{sx\sigma(s)}(b) - d_{sx}(b) = \frac{1}{2} \Big( \ell(sx\sigma(s))+ \ell(\eta(sx\sigma(s))) - \ell(sx) - \ell(\eta(sx))\Big),
\end{equation*}
which is a constant independent of the class $[b]$, and thus is positive.
\end{proof}

We now focus on a case where all dimensions of irreducible components of $X_{sx}(b)$ and $X_{sx\sigma(s)}(b)$ can be computed. In the next section we show that these conditions can indeed be satisfied.

\begin{prop}\label{lem:bgne}
Let $(x,s)$ be as in Lemma \ref{lem:neq} and such that both $sx$ and $sx\sigma(s)$ are cordial. Let $[b_{sx}]$ and $[b_{sx\sigma(s)}]$ be the generic classes of $IsxI$ and $Isx\sigma(s)I$. Then 
\begin{equation}\label{eq:bg}
     B(G)_ x = B(G)_{sx} \cup B(G)_{sx\sigma(s)} = \{[b] \in B(G) \mid [b] \le [b_{sx}] \text{ or } [b] \le [b_{sx\sigma(s)}]\},
\end{equation} and $[b_x] = \max\{[b_{sx}], [b_{sx\sigma(s)}]\}$.
Denote by $B(G)^{\noneq}_x$ the set of classes of $B(G)_x$ whose Newton strata in $IxI$ have irreducible components of different codimensions. Then
\begin{equation}\label{eq:ne}
    B(G)^{\noneq}_x = B(G)_{sx} \cap B(G)_{sx\sigma(s)} = \{[b] \in B(G) \mid[b] \le [b_{sx}] \text{ and } [b] \le [b_{sx\sigma(s)}]\}.
\end{equation} Moreover, for any $b \in B(G)^{\noneq}_x$ the variety $X_x(b)$ has dimension \begin{equation*}
   \dim X_x(b) =  d_{sx\sigma(s)}(b) +1.
\end{equation*}
\end{prop}  
\begin{proof}
By the reduction method of \cite{he}*{Prop.~4.2}, $X_x(b)$ is non-empty if and only if $X_{sx}(b)$ or $X_{sx\sigma(s)}(b)$ is non-empty. 

Since $sx$ and $sx\sigma(s)$ are cordial, all non-empty affine Deligne-Lusztig varieties for $sx$ or $sx\sigma(s)$ are equidimensional of the respective virtual dimension. Thus by Lemma \ref{lem:neq}(ii) and (iii), $X_x(b)$ is non-equidimensional if and only if $X_{sx}(b)$ and $X_{sx\sigma(s)}(b)$ are non-empty. 

In the proof of Lemma \ref{lem:neq}(iii) we saw that the basic Newton stratum in $IsxI$ is non-empty. By \cite{Vimi}*{Thm.~1.1 and Rmk. 3.18} $B(G)_{sx}$ is saturated and hence equal to the set of all classes $[b]\leq [b_{sx}]$, and similarly for  $sx\sigma(s)$. 

The statement on the dimension of $X_x(b)$ directly follows from Lemma \ref{lem:neq} (i) and (iii) and cordiality of $sx$ and $sx\sigma(s)$.
\end{proof}

\subsection{An explicit example}
In order to find a pair $(x,s)$ as in Proposition \ref{lem:bgne}, we use the mathematics software system 
\verb+SageMath+\footnote{{www.sagemath.org}}. Here, to compute generic Newton points and to check cordiality we use the results of \cite{Mi}*{Thm.~3.2} and \cite{Vimi}*{Prop.~4.2}, respectively. The results of \cite{Mi} offer a very useful description of the generic Newton point for split groups $G$ and any $x \in \widetilde{W}$ of the form $x=vt^{\mu}w$ where $\langle\alpha,\mu\rangle>M$ for all simple roots $\alpha$ where $M$ is a constant depending on $G$, $v$ and $w$. Under the same assumption, \cite{Vimi} gives a criterion to check if $x$ is cordial in terms of some paths in the quantum Bruhat graph. We include the \verb+SageMath+ code in the appendix.

For $G$ of type $A_3$, an exhaustive computer search (on the finitely many possible triples $(v,w,s)$) shows that there are no pairs $(x, s)$ as in Proposition \ref{lem:bgne}. However, for type $A_4$ there are many such pairs. The following example is a particularly suitable one in the sense that the two classes $[b_{sx}]$ and $[b_{sxs}]$ (which correspond to the maximal non-empty and the maximal non-equicodimensional Newton stratum in $IxI$) are close to each other in $B(G)$.

\begin{ex}\label{ex:1}
Let $G$ be of type $A_4$, $x = vt^{\mu}w$ in the affine Weyl group with
\begin{equation*}
    v = s_4s_2s_3s_1, \quad w = s_1s_2s_3s_4s_2s_3s_1,  \quad \mu = (150, 75, 0 -75, -150), \quad s = s_2
\end{equation*} where $s_i$ denote the simple reflections in $\mathbb{S}$. In this example, the constant $M$ for the regularity assumption on $\mu$ equals 74 and is thus satisfied. Then the pair $(x, s)$ satisfies the requirements of Proposition \ref{lem:bgne} and the superregularity hypothesis of \cite{Mi} and \cite{Vimi}.

By the description in (\ref{eq:ne}) of the set of classes whose Newton strata are not equicodimensional, the next step is to find the generic Newton points associated to $sx$ and $sxs$. By \cite{Mi}*{Thm.~3.2} one obtains 
\begin{equation*}
    \nu_{sx} = (149, 75, 0, -75, -149), \quad
    \nu_{sxs} = (149, 74, 0, -74, -149), \quad
    \nu_{x} = \nu_{sx}
\end{equation*} 
where $\nu_x$, $\nu_{sx}$ and $\nu_{sxs}$ denote the generic Newton points of $IxI$, $IsxI$ and $IsxsI$, respectively. Observe that $\nu_x = \nu_{sx} > \nu_{sxs}$ in the dominance order, and  \begin{equation*}\label{eq.diff}
    \nu_{x} - \nu_{sxs} = \alpha_2^{\vee} + \alpha_3^{\vee}
\end{equation*}
where $\alpha_i^{\vee}$ denotes the coroot $e_i-e_{i+1}$. 

From $\nu_{sx} > \nu_{sxs}$ and Proposition \ref{lem:bgne} we obtain that $$B(G)_x=\big\{[b] \in B(G)\mid \nu_b\leq\nu_x=(149, 75, 0, -75, -149)\big\}.$$ Moreover, by (\ref{eq:ne}), the subset of $[b]\in B(G)_x$ whose Newton stratum  is not equicodimensional in $IxI$ is exactly the set $\{[b] \in B(G) \mid [b] \le [b_{sxs}]\}$. In particular, $\mathcal N_{[b_{sxs}],x}\subset IxI$ in non-equicodimensional and is maximal with this property.

There are two maximal chains in $B(G)_x$ from $[b_{sxs}]$ to $[b_x]$. They are given by $ [b_{sxs}] < [b_i] < [b_x]$ where $[b_i]$ (for $i = 2, 3$) is the class with Newton point $\nu_i=\nu_x-\alpha_i^{\vee}$. 

Let $n\in \mathbb N$ be such that all Newton strata in $IxI$ are invariant under $I_n$. Let $Z$ an irreducible component of codimension $> 1$ of $\N_{[b_{sxs}],x}/I_n$, which exists since this stratum is not equidimensional and has codimension $1$. Let $U_2$ be the open subscheme of $IxI/I_n$ where $\proj_{(2)}(\nu(z)) = \proj_{(2)}(\nu_x)$ (i.e., the sum of the first two entries of $\nu(z)$ is 224), and let $S_2$ be its complement. Then $S_2$ contains $Z$ and $\mathcal{N}_{[b_i],x}/I_n$ for $i = 3$. By the Purity Theorem, $S_2$ is pure of codimension $1$ in $IxI/I_n$. Up to considering an irreducible component containing $Z$, we can assume that $S_2$ is irreducible. Since $Z$ has codimension at least $1$ in $S_2$, its generic Newton class $[b_{sxs}]$ is then smaller than that of $S_2$. By definition of $S_2$, it follows that the generic class of $S_2$ can only be $[b_3]$, and therefore, $Z$ is contained in the closure of $\mathcal{N}_{b_3}$. With a similar argument, we find that any irreducible component of $\mathcal{N}_{[b_{sxs}], x}$ of codimension $>1$ is contained in the closure of $\mathcal{N}_{[b_2], x}$, as well.

The Newton stratum $\mathcal{N}_{[b_{sxs}], x}$ also has an irreducible component of codimension $1$, which is thus not contained in the closure of any other Newton stratum except for $\N_{[b_x],x}$. Hence the closures of $\mathcal{N}_{[b_i], x}$ for $i= 2,3$ each have a non-empty intersection with $\N_{[b_{sxs}], x}$ but do not contain it. 
\end{ex}

\section*{Appendix}
The following script can be run in \verb+SageMath+ to compute all pairs $(x,s)$ as in Proposition \ref{lem:bgne} for $G$ of type $A_4$. By changing the parameters in the first line, one can see, for example, that for $G$ of type $A_3$ no such pair exists. 

\begin{verbatim}

W = WeylGroup(['A', 4], prefix = 's')
S = W.simple_reflections()
qbg = W.quantum_bruhat_graph()
for v in W:
    for w in W:
        for s in S:
            if ( (s*v).length() < v.length() and 
            (w*s).length() > w.length() and 
            (w*s*v).length() < (w*v).length() -1 and
            (w*v).has_full_support() and
            (w*s*v).has_full_support() and 
            qbg.distance(w^-1, s*v) == (w*s*v).length() and
            qbg.distance(s^-1*w^-1, s*v) == (w*v).length() ) :
                print("(%s, %s, %s), sx and sxs are cordial" 
                % (v, w, s))
\end{verbatim}

For $G$ of type $A_4$ we obtain the following list of triples $(v,w,s)$ satisfying the assumptions of Proposition \ref{lem:bgne}.

\begin{longtable}{|L|L|L||L|L|L|} 
\hline
v & w & s & v & w & s \\
\hline
s_3s_2 & s_2s_3s_4s_1s_2 & s_3 & s_3s_4s_2 & s_2s_3s_4s_3s_1s_2 & s_3\\
\hline

s_3  s_4  s_2  s_1 & s_1  s_2  s_3  s_4  s_3  s_1  s_2 & s_3 & 
 s_2   s_3 &  s_3   s_4   s_2   s_3   s_1 &  s_2    \\ 
\hline
s_2   s_3   s_1 &  s_3   s_4   s_1   s_2   s_3   s_1 &  s_2 &
 s_2   s_3   s_4   s_1 &  s_3   s_4   s_1   s_2   s_3   s_2   s_1 &  s_2    \\ \hline
 s_4   s_2   s_3 &  s_3   s_4   s_2   s_3   s_1 &  s_2 & 
 s_4   s_2   s_3 &  s_2   s_3   s_4   s_2   s_3   s_1 &  s_2   \\ \hline
 s_4   s_2   s_3 &  s_4   s_2   s_3   s_1 &  s_2 & 
 s_2   s_3   s_4   s_3 &  s_2   s_3   s_1 &  s_2 \\ \hline
s_4   s_2   s_3   s_1 &  s_1   s_2   s_3   s_4   s_2   s_3   s_1 &  s_2 &
 s_2   s_3   s_4   s_3   s_1 &  s_1   s_2   s_3   s_1 &  s_2   \\ \hline
s_2   s_3   s_4   s_3   s_1 &  s_3   s_4   s_1   s_2   s_3   s_2   s_1 &  s_2 &
 s_2   s_3   s_4   s_3   s_1 &  s_1   s_2   s_3   s_4   s_2   s_3   s_2   s_1 &  s_2   \\ \hline
s_4   s_2   s_3   s_1   s_2   s_1 &  s_3   s_4   s_2   s_3   s_1 &  s_2 &
 s_4   s_2   s_3   s_1   s_2   s_1 &  s_4   s_2   s_3 &  s_4   \\ \hline
 s_2   s_3   s_4   s_3   s_1   s_2 &  s_1   s_2   s_3 &  s_4 &
 s_2   s_3   s_4   s_3   s_1   s_2 &  s_3   s_4   s_1   s_2   s_3   s_2   s_1 &  s_2  \\ \hline
  s_2   s_3   s_4   s_3   s_1   s_2   s_1 &  s_2   s_3 &  s_4   &
  s_2   s_3   s_4   s_3   s_1   s_2   s_1 &  s_3   s_4   s_2   s_3   s_2   s_1 &  s_2     \\ \hline
  s_3   s_1   s_2 &  s_2   s_3   s_4   s_1   s_2 &  s_3   &
  s_3   s_1   s_2 &  s_2   s_3   s_4   s_1   s_2   s_1 &  s_3     \\ \hline
  s_3   s_1   s_2 &  s_3   s_4   s_1   s_2 &  s_3   &
  s_3   s_1   s_2   s_1 &  s_3   s_4   s_2 &  s_3     \\ \hline
  s_3   s_4   s_1   s_2 &  s_2   s_3   s_4   s_3   s_1   s_2   s_1 &  s_3   &
  s_3   s_4   s_1   s_2   s_1 &  s_3   s_4   s_3   s_2 &  s_3     \\ \hline
  s_3   s_4   s_1   s_2   s_1 &  s_1   s_2   s_3   s_4   s_3   s_1   s_2 &  s_3   &
  s_3   s_4   s_1   s_2   s_1 &  s_1   s_2   s_3   s_4   s_3   s_1   s_2   s_1 &  s_3     \\ \hline
  s_3   s_4   s_1   s_2   s_3   s_1 &  s_4   s_3   s_2 &  s_1   &
  s_3   s_4   s_1   s_2   s_3   s_1 &  s_1   s_2   s_3   s_4   s_3   s_1   s_2 &  s_3     \\ \hline
  s_1   s_2   s_3   s_4   s_2   s_3 &  s_2   s_3   s_4   s_1   s_2 &  s_3   &
  s_1   s_2   s_3   s_4   s_2   s_3 &  s_3   s_1   s_2 &  s_1     \\ \hline
  s_1   s_2   s_3   s_4   s_2   s_3   s_1 &  s_3   s_2 &  s_1   &
  s_1   s_2   s_3   s_4   s_2   s_3   s_1 &  s_1   s_2   s_3   s_4   s_1   s_2 &  s_3     \\ \hline
  s_3   s_4   s_1   s_2   s_3   s_2 &  s_4   s_1   s_2   s_3   s_2 &  s_1   &
  s_3   s_4   s_1   s_2   s_3   s_2   s_1 &  s_4   s_3   s_2 &  s_1     \\ \hline
  s_3   s_4   s_1   s_2   s_3   s_2   s_1 &  s_3   s_4   s_3   s_2 &  s_1   &
  s_3   s_4   s_1   s_2   s_3   s_2   s_1 &  s_4   s_2   s_3   s_2 &  s_1     \\ \hline
  s_4   s_1   s_2   s_3   s_2 &  s_4   s_3   s_1   s_2 &  s_1   &
  s_4   s_1   s_2   s_3   s_2 &  s_4   s_1   s_2   s_3 &  s_4     \\ \hline
  s_4   s_1   s_2   s_3   s_2   s_1 &  s_4   s_3   s_2 &  s_1   &
  s_4   s_1   s_2   s_3   s_2   s_1 &  s_4   s_2   s_3 &  s_4     \\ \hline
  s_1   s_2   s_3   s_4   s_3   s_2 &  s_3   s_1   s_2 &  s_1   &
  s_1   s_2   s_3   s_4   s_3   s_2 &  s_1   s_2   s_3 &  s_4     \\ \hline
  s_1   s_2   s_3   s_4   s_2   s_3   s_2 &  s_3   s_1   s_2 &  s_1   &
  s_1   s_2   s_3   s_4   s_2   s_3   s_2 &  s_3   s_4   s_1   s_2 &  s_1     \\ \hline
  s_1   s_2   s_3   s_4   s_2   s_3   s_2 &  s_1   s_2   s_3   s_2 &  s_1   &
  s_1   s_2   s_3   s_4   s_2   s_3   s_2   s_1 &  s_3   s_4   s_2 &  s_1     \\ \hline
  s_4   s_1   s_2   s_3   s_1   s_2 &  s_4   s_1   s_2   s_3   s_2 &  s_4   &
  s_4   s_1   s_2   s_3   s_1   s_2   s_1 &  s_4   s_2   s_3 &  s_4     \\ \hline
  s_4   s_1   s_2   s_3   s_1   s_2   s_1 &  s_4   s_2   s_3   s_2 &  s_4   &
  s_4   s_1   s_2   s_3   s_1   s_2   s_1 &  s_4   s_2   s_3   s_1 &  s_4     \\ \hline
  s_1   s_2   s_3   s_4   s_3   s_1   s_2 &  s_1   s_2   s_3 &  s_4   &
  s_1   s_2   s_3   s_4   s_3   s_1   s_2 &  s_1   s_2   s_3   s_2 &  s_4     \\ \hline
  s_1   s_2   s_3   s_4   s_3   s_1   s_2 &  s_1   s_2   s_3   s_1 &  s_4   &
  s_1   s_2   s_3   s_4   s_3   s_1   s_2   s_1 &  s_2   s_3   s_1 &  s_4     \\ \hline

\end{longtable}


\begin{bibdiv}[GHN]
  \begin{biblist}*{labels={shortalphabetic}}
  
  \bib{B}{article}{
 author = {Beazley, Elizabeth T.} ,
 journal = {Math. Zeitschrift},
 number = {3},
 pages = {499--540},
 title = {Codimensions of Newton strata for SL(3) in the Iwahori case},
 volume = {263},
 year = {2009},
 label = {B},
 }

  
 \bib{chai}{article}{
 author = {Chai, Ching-Li} ,
 journal = {American Journal of Mathematics},
 number = {5},
 pages = {967--990},
 title = {Newton Polygons as Lattice Points},
 volume = {122},
 year = {2000},
 label = {C},
 }
 

 
 \bib{gor}{article}{
 author={G\"ortz, Ulrich},
  author={He, Xuhua},
 title={Dimension of affine Deligne-Lusztig varieties in affine flag varieties},
 journal= {Documenta Mathematica},
  year={2010},
  volume = {15},
  pages = {1009--1028}
 }
 
 \bib{GHN}{article}{
 author={G\"ortz, Ulrich},
  author={He, Xuhua},
  author={Nie, Sian}
 title={P-alcoves and nonemptiness of affine Deligne-Lusztig varieties},
 journal= {Ann. Sci. \'Ecole Norm. Sup.},
  year={2015},
  volume = {48},
  pages = {647--665}
 }
 

 \bib{ham}{article}{
 author={Hamacher, Paul},
 title={The almost product structure of Newton strata in the Deformation space of a Barsotti-Tate group with crystalline Tate tensors},
  journal={Mathematische Zeitschrift},
  volume = {287},
  pages={1255--1277},
  year={2017},
  label = {Ha}
 }
 
 
 \bib{he}{article}{
 author={He, Xuhua},
 title={Geometric and homological properties of affine Deligne-Lusztig varieties},
  journal={Annals of Mathematics},
  volume = {179},
     number = {1},
  pages = {367--404},
  year = {2014},
  label = {He1}
 }
 
 \bib{he2}{article}{
 author={He, Xuhua},
 title={Hecke algebras and $p$-adic groups},
  journal={Current developments in mathematics},
  pages = {73--135},
  year = {2016},
  publisher={International Press},
  address = {Somerville, MA},
  label = {He2}
 }
 
 \bib{HeNote}{article}{
 author={He, Xuhua},
 title={Note  on  affine  Deligne-Lusztig  varieties},
  journal={Proceedings  of  the  Sixth  International Congress of Chinese Mathematicians.  Vol.  I, Adv.  Lect.  Math.  (ALM) },
  volume = {36}
  pages = {297--307},
  year = {2017},
  publisher={International Press},
  address = {Somerville, MA},
  label = {He3}
 }
 
 \bib{He20}{article}{
 author={He, Xuhua},
 title={Cordial elements and dimensions of affine Deligne-Lusztig varieties},
  journal={preprint, arxiv:2001.03325},
  label = {He4}
 }
 
 
 
 
 \bib{Ko}{article}{
 author={Kottwitz, Robert E.},
 title={Isocrystals with additional structure},
  journal={Compositio Mathematica},
  volume = {56},
     number = {2},
  pages={201--220},
  year={1985},
  label = {Ko1}
 }
 
 
 \bib{Ko''}{article}{
 author={Kottwitz, Robert E.},
 title={Isocrystals with additional structure. II},
  journal={Compositio Mathematica},
  volume = {109},
  pages={255--339},
  year={1997},
  label = {Ko2}
 }
 
 \bib{kot6}{article}{
 author={Kottwitz, Robert E.},
 title={Dimensions of Newton strata in the adjoint quotient of reductive groups},
  year={2006},
  journal = {Pure and Applied Mathematics Quarterly},
  volume = {2},
  number = {3},
  pages = {817--836},
    label= {Ko3}
 }
 

\bib{Mi}{article}{
 author={Milićević, Elizabeth},
 title={Maximal Newton Points and the Quantum Bruhat Graph},
  journal={Michigan Mathematical Journal},
  pages = {forthcoming},
  year={2021},
  label = {M}
 }
 
\bib{Vimi}{article}{
 author={Milićević, Elizabeth},
    author={Viehmann, Eva},
 title={Generic Newton points and the Newton poset in Iwahori-double cosets},
  journal={Forum of Mathematics, Sigma},
  volume = {8},
  pages = {to appear},
  year = {2020}
 }


\bib{rap}{article}{
 author={Rapoport, Michael},
 title={A guide to the reduction modulo $p$ of Shimura Varieties},
 journal = {Ast\'erisque},
 number = {298},
    year = {2005},
    pages = {271--318},
 series = {Automorphic forms I},
 label = {R}
 }

\bib{RR}{article}{
 author={Rapoport, Michael},
    author={Richartz, Melanie},
 title={On the classification and specialization of $F$-isocrystals with additional structure},
  journal={Compositio Mathematica},
  volume = {103},
     number = {2},
  pages = {153--181},
  year = {1996}
 }
 
\bib{V13}{article}{
 author={Viehmann, Eva},
 title={Newton strata in the loop group of a reductive group},
 journal = {American Journal of Mathematics},
 number = {135},
    year = {2013},
    pages = {499--518},
 }



  \end{biblist}
\end{bibdiv}
\end{document}